\documentclass[12pt, reqno]{amsart}
\usepackage{fullpage}
\usepackage{amsmath,amssymb,amsfonts,amsthm}
\usepackage{tikz}
\usepackage{enumitem}
\usepackage{graphicx}
\usepackage{pgfplots}
\usepackage{thm-restate}
\usepackage{leftindex}

\usepackage[colorlinks=true,linkcolor=blue, allcolors=blue]{hyperref}

\usepackage[backend=biber, style=alphabetic, sorting=nyt, isbn=false, url=false, doi=false, maxnames=5, maxalphanames=5, maxcitenames=5]{biblatex}
\makeglossary
\usepackage{bbm}
\usepackage{cleveref}
\usepackage{MnSymbol}
\usepackage{csquotes}
\usepackage{enumitem}
\usepackage{verbatim}
\usepackage{tikz}
\usepackage{tikz-cd}
\usepackage{color}
\usepackage{graphicx}

\newcommand{\f}{\mathcal{F}}

\newcommand{\nn}{\mathcal{N}}

\newcommand{\bC}{\mathbb{C}}

\newcommand{\cE}{\mathcal{E}}
\newcommand{\cG}{\mathcal{G}}

\newcommand{\cN}{\mathcal{N}}
\newcommand{\cO}{\mathcal{O}}
\newcommand{\fX}{\mathfrak{X}}
\newcommand{\ca}[1]{\mathcal{#1}}

\newcommand{\cha}{\textrm{char}}
\newcommand{\Perv}{\mathrm{Perv}}

\newcommand{\remove}[1]{ }

\newcommand{\fg}{\mathfrak{g}}

\newcommand{\en}{\mathrm{End}}

\newcommand{\bZ}{\mathbb{Z}}
\newcommand{\bk}{\mathbbm{k}}

\newcommand{\maps}{\rightarrow}

\newcommand{\ind}{\mathrm{Ind}}
\newcommand{\res}{\mathrm{Res}}
\newcommand{\pres}{{}' \mathrm{Res}}
\newcommand{\Hom}{\mathrm{Hom}}

\newcommand{\cusp}{\mathrm{cusp}}

\newcommand{\Rep}{\mathrm{Rep}}
\newcommand{\Stab}{\mathrm{Stab}}
\newcommand{\simto}{\xrightarrow{\sim}}
\newcommand{\Db}{D^{\mathrm{b}}}
\newcommand{\fR}{\mathfrak{R}}
\newcommand{\fL}{\mathfrak{L}}
\newcommand{\IC}{\mathcal{IC}}
\newcommand{\rZ}{\mathrm{Z}}
\newcommand{\rN}{\mathrm{N}}
\newcommand{\rA}{\mathrm{A}}

\newtheorem{theorem}{Theorem}[section]
\newtheorem{lemma}[theorem]{Lemma}
\newtheorem{definition}[theorem]{Definition}

\newtheorem{proposition}[theorem]{Proposition}
\newtheorem{corollary}[theorem]{Corollary}

\theoremstyle{remark}
\newtheorem{remark}[theorem]{Remark}
\newtheorem{example}[theorem]{Example}

\renewbibmacro{in:}{
	\ifentrytype{incollection}
	{\printtext{\bibstring{in}\intitlepunct}}
	{}
}

\pgfplotsset{compat=1.16}

\title{On the modular generalized Springer Correspondence for disconnected groups}

\author{Kostas I. Psaromiligkos}
\address{Universit\'e Clermont Auvergne, CNRS, LMBP, F-63000 Clermont-Ferrand, France.}
\email{konstantinos.psaromiligkos@uca.fr}

\author{Simon Riche}
\address{Universit\'e Clermont Auvergne, CNRS, LMBP, F-63000 Clermont-Ferrand, France.}
\email{simon.riche@uca.fr}

\date{}
\keywords{} 

\begin{document}
	
	\begin{abstract}
	We study the construction of a modular generalized Springer correspondence for a possibly disconnected complex reductive algebraic group.
	\end{abstract}
	
	\maketitle
	
	
	\section{Introduction}

	\subsection{Springer correspondences}

The Springer correspondence is a very classical and important geometric construction that relates equivariant perverse sheaves on the nilpotent cone of a complex connected reductive algebraic group and representations of its Weyl group. (In both cases, the coefficient field $\bk$ is a field of characteristic $0$, at least in the initial version.) This subject was invented by Springer~\cite{springer}, but it has seen many useful reformulations, the one most relevant for us by Lusztig~\cite{lusztig-green}. 

There are by now many generalizations of this construction. The ones relevant for this paper are first the \emph{generalized Springer correspondence} due to Lusztig~\cite{Lusztig1984}, which provides a way to ``complete'' the injection constructed by Springer into a bijection between appropriate sets, and secondly the adaptation of Lusztig's methods to the setting where $\bk$ can possibly have positive characteristic, see~\cite{Achar2016a,Achar2017,Achar2017a}. 

The main goal of the present paper is to provide the first steps towards a further generalization of these results, that allows $G$ to be a \emph{possibly disconnected} reductive algebraic group. This setting, but with $\bk$ of characteristic $0$, has already been studied in the literature by Aubert--Moussaoui--Solleveld~\cite{Aubert2018} and Dillery--Schwein~\cite{Dillery2024}, with (partly conjectural) applications to the Langlands correspondence for representations of $p$-adic groups, from which we draw a lot of inspiration.

\begin{remark}
Another form of the (generalized) Springer correspondence for disconnected reductive groups was constructed earlier in~\cite{lusztig-char-sheaves,sorlin}, with applications to representations of finite groups of Lie type. In this version it is important to work in the setting where $G$ is defined over a field of positive characteristic, to allow the unipotent variety to be possibly disconnected, but one mainly considers the action of $G^\circ$. These two directions are unrelated as far as we understand, and the one we follow is that of~\cite{Aubert2018}.
\end{remark}

	\subsection{Statement}

	Let $G$ be a possibly disconnected complex reductive algebraic group, and consider its identity component $G^{\circ}$, and their common Lie algebra $\fg$. Let $\nn_G\subseteq \fg$ be the nilpotent cone of $G$. We consider the category $\Perv_G(\nn_G,\bk)$ of $G$-equivariant perverse sheaves on $\nn_G$ with coefficients in an algebraically closed\footnote{In the body of the paper we in fact consider a general field of coefficients, but we assume it is algebraically closed in this introduction to avoid some irrelevant subtleties related to fields of definition of representations.} field $\bk$. 

	Note that the nilpotent cones of $G$ and $G^\circ$ coincide; the difference with the case of connected groups therefore only lies in the equivariance condition.
	
	By the general theory of perverse sheaves, the isomorphism classes of simple objects of $\Perv_G(\nn_G,\bk)$ are in a canonical bijection with the set $\fR_{G,\bk}$ of equivalence classes of pairs $(\ca O, \ca E)$, where $\ca O \subseteq \nn_G$ is a $G$-orbit and $\ca E$ is an irreducible $G$-equivariant $\bk$-local system on $\ca O$, and the goal of the generalized Springer correspondence is to describe this set in terms of simple representations of appropriate generalizations of the Weyl group of $G^\circ$. (The equivalence relation considered here corresponds to isomorphisms of local systems.) 

More specifically, we have ``induction'' and ``restriction'' functors that relate $\fR_{G,\bk}$ with its analogues for Levi subgroups, which give rise to a notion of cuspidality. Then the statement can be broken into two steps:
\begin{enumerate}
\item
\label{it:intro-partition}
\emph{Partition into induction series}. Here one constructs a partition of $\fR_{G,\bk}$ into subsets parametrized by $G$-conjugacy classes of ``cuspidal triples,'' i.e.~triples $(L,\cO,\cE)$ where $L$ is a Levi subgroup of $G$ and $(\cO,\cE)$ is a pair in $\fR_{L,\bk}$ which is cuspidal.
\item
\label{it:intro-parametrization}
\emph{Parametrization of induction series}. Here, given a cuspidal triple $(L,\cO,\cE)$, one constructs a bijection between the associated subset of $\fR_{G,\bk}$ and some class of simple representations of an appropriate ``generalized Weyl group'' attached to $G$ and $(L,\cO,\cE)$.
\end{enumerate}

In the setting of this paper, we provide a satisfactory version of~\eqref{it:intro-partition} in Theorem~\ref{disj}. Essentially the only subtlety here is to choose the appropriate notion of ``Levi subgroup'' for a disconnected reductive algebraic group. But this question has already been tackled in~\cite{Aubert2018}, which introduced a notion of \emph{special} Levi subgroups.\footnote{Our terminology differs from that adopted in~\cite{Aubert2018}, where such Levi subgroups are called quasi-Levi subgroups; see Footnote~\ref{fn:Levi}.} (Note that there is some flexibility in the choice of the class of Levi subgroups. We believe the class of special Levi subgroups is the most natural one, but our methods apply for some other classes, see Remark~\ref{rmk:parabolics}.) The partition of $\fR_{G,\bk}$ is then obtained in a way completely parallel to that considered in~\cite{Achar2017a} in the connected case, or~\cite{Dillery2024} for characteristic-$0$ coefficients. It is based on a ``Mackey formula'' (sometimes called ``geometrical lemma''), for which the geometric input is developed in sufficient generality in~\cite{Dillery2024} to be directly applied here.

Our version of~\eqref{it:intro-parametrization} is slightly less satisfactory. This is related to a true additional difficulty that arises in the disconnected setting, already observed in~\cite{Aubert2018}. Namely, in the connected case, one constructs a canonical bijection between the part of $\fR_{G,\bk}$ corresponding to $(L,\cO,\cE)$ and the isomorphism classes of simple $\bk$-representations of the finite group $\rN_G(L)/L$. A first (minor) difficulty in the disconnected setting is that the latter group should be replaced by its subgroup defined as the stabilizer of the pair $(\cO,\cE)$, see~\S\ref{ss:Weyl-gps}. (In the connected case this stabilizer is always the whole of $\rN_G(L)/L$, but this is no longer true in the disconnected case.) A more important difficulty is that representations should be replaced by modules over a certain twisted group algebra, in which the twisting can \emph{really} be nontrivial. (See~\S\ref{ss:example-cocyle} for a concrete example, reproduced from~\cite{Aubert2018}.) The involved cocycle is still rather mysterious, even in the setting of characteristic-$0$ coefficients, and unfortunately our results do not contribute to a clarification of this question. What we do is constructing a canonical bijection between the part of $\fR_{G,\bk}$ associated with $(L,\cO,\cE)$ and simple modules over a certain algebra, which we show is isomorphic (but not in a particularly canonical way) to a twisted group algebra of the stabilizer of $(\cO,\cE)$ in $\rN_G(L)/L$; see Theorem~\ref{thm:parind}.

Improving the answer to~\eqref{it:intro-parametrization} should be the next step in the development of the generalized Springer correspondence for disconnected groups.

	\subsection{Outline}

In Section~\ref{sec:preliminaries} we gather some generalities related to Clifford theory for finite groups. In Section~\ref{sec:definitions} we define the notion of cuspidality, and the induction series. The definition of these series involves choices of parabolic subgroups that have the corresponding Levi subgroups as Levi factors, and one needs to prove that they are in fact independent of these choices; this fact is proved in~\S\ref{ss:Fourier} following the same strategy as for the connected case in~\cite{Achar2016a}, whose main ingredient is a Fourier transform. In Section~\ref{sec:partition} we prove a Mackey formula, and deduce the fact that induction series provide a partition of $\fR_{G,\bk}$. Finally, in Section~\ref{sec:parametrization} we explain what can be said about the parametrization of induction series; our method here is to translate the problem in terms of finite groups, where we apply the results of Section~\ref{sec:preliminaries}.

	\subsection{Acknowledgements}
	
We thank Anne-Marie Aubert, Peter Dillery and David Schwein for their interest and helpful discussions, and Abel Lacabanne for his help with the material of~\S\ref{ss:simple-subobjs-ind-rep}. We also thank Maarten Solleveld for the useful correspondence, and in particular for pointing out a mistake in an earlier version of this paper and showing us Example~\ref{ex:weyl}.
	
This project has received funding from the European Research Council (ERC) under the European Union’s Horizon 2020 research and innovation programme (grant agreement No.~101002592).

	\section{Preliminaries on representations of finite groups}
	\label{sec:preliminaries}

In this section we collect some elementary results on representation theory of finite groups. The proofs are easy exercises in Clifford theory.

\subsection{Induction and restriction from normal subgroups}
\label{ss:ind-res-normal}

We fix a field $\bk$. For any finite group $K$, we will denote by $\Rep_\bk(K)$ the category of finite-dimensional representations of $K$ over $\bk$.

If $H$ is a finite group and $N \subseteq H$ a subgroup,
we have an exact restriction functor
\[
\res^H_N : \Rep_\bk(H) \to \Rep_\bk(N).
\]
We also have an exact induction functor
\[
\ind_N^H := \bk H \otimes_{\bk N} (-) : \Rep_\bk(N) \to \Rep_\bk(H)
\]
which is both left and right adjoint to $\res^H_N$, see~\cite[(10.8) and (10.21)]{Curtis1990}.

In practice we will assume that $N$ is normal. In this case we have the following statements about the behavior of these functors with respect to semisimplicity.

\begin{lemma}
\phantomsection
\label{lem:ind-res-ss-gps}
\begin{enumerate}
\item
\label{it:res-ss-gps}
If $V \in \Rep_\bk(H)$ is simple, then $\res^H_N(V)$ is semisimple.
\item
\label{it:ind-ss-gps}
If $V \in \Rep_\bk(N)$ is simple and $\mathrm{char}(\bk)$ does not divide $|H/N|$, then $\ind^H_N(V)$ is semisimple.
\end{enumerate}
\end{lemma}

\begin{proof}
\eqref{it:res-ss-gps}
This is a classical result of Clifford, see~\cite[Theorem~11.1]{Curtis1990}.

\eqref{it:ind-ss-gps}
This is also classical, see e.g.~\cite[Chap.~2, Proposition~2.13]{karpilovsky2}.
\end{proof}

\begin{corollary}
\label{cor:subobjs-quotients}
Let $V \in \Rep_\bk(N)$. Then the sets of isomorphism classes of simple subobjects and of simple quotients of $\ind^H_N(V)$ coincide; in fact they both consist of the simple $H$-modules $V'$ such that $\res^H_N(V')$ contains $V$ as a direct summand.
\end{corollary}

\begin{proof}
A simple representation $V'$ of $H$ is isomorphic to a subobject, resp.~quotient, of $\ind^H_N(V)$, if and only if $\Hom_H(V', \ind^H_N(V)) \neq 0$, resp.~$\Hom_H(\ind^H_N(V), V') \neq 0$. Hence the claim follows from (bi)adjunction and the semisimplicity of $\res^H_N(V')$, see Lemma~\ref{lem:ind-res-ss-gps}\eqref{it:res-ss-gps}.
\end{proof}

\subsection{Simple subobjects of induced representations}
\label{ss:simple-subobjs-ind-rep}

We continue to consider a normal subgroup $N$ of a finite group $H$, and a field $\bk$. Our goal is to obtain a parametrization of the sets appearing in Corollary~\ref{cor:subobjs-quotients}.

If $V \in \Rep_\bk(N)$ and $g \in H$ we have the $N$-module ${}^g V$ whose underlying vector space is $V$, and the action $\cdot^g$ of $N$ is defined by $h \cdot^g v = (g^{-1} h g) \cdot v$ for $v \in V$ and $h \in N$. Up to isomorphism, this representation only depends on the coset $gN$. Therefore, the subgroup $\Stab_H(V)=\{h \in H \mid {}^h V \cong V\}$ contains $N$. It is a standard observation that, in this setting, if $S \subset H$ is a subset of representatives for the cosets in $H/N$, then we have a canonical isomorphism
\begin{equation}
\label{eqn:decomp-res-ind}
\res^H_N(\ind_N^H(V)) \cong \bigoplus_{s \in S} {}^s V.
\end{equation}
More specifically, as vector spaces we have $\ind_N^H(V) = \bigoplus_{s \in S} \bk \cdot s \otimes V$, and $\bk \cdot s \otimes V$ is an $N$-stable subspace which identifies with ${}^s V$. More generally, if $K \subseteq H$ is a subgroup containing $N$ and $V \in \Rep_\bk(K)$, if $S \subset H$ is a subset of representatives for the cosets in $H/K$, then we have a canonical isomorphism
\begin{equation}
\label{eqn:decomp-res-ind-2}
\res^H_N(\ind_K^H(V)) \cong \bigoplus_{s \in S} {}^s (\res^K_N(V)).
\end{equation}

The following lemma will serve as an ingredient in the proof of Proposition~\ref{prop:bijection-simple-subreps} below.
	
\begin{lemma}
\label{lem:subreps-stabilizer}

Let $V$ be a simple $N$-module, and
set $K=\Stab_H(V)$. Then the assignment $V' \mapsto \ind_K^H(V')$ induces a bijection between the sets of isomorphism classes of simple $K$-subrepresentations of $\ind_N^K(V)$ and of simple $H$-subrepresentations of $\ind_N^H(V)$.
\end{lemma}

\begin{proof}
First we observe that the assignment $V' \mapsto \ind_K^H(V')$ induces an injective map from the set of isomorphism classes of simple $K$-representations which, as $N$-modules, are direct sums of copies of $V$, to the set of isomorphism classes of simple $H$-modules. In fact, let $V'$ be a simple $K$-representation such that $\res^K_N(V')$ is a direct sum of copies of $V$, and let $V'' \subset \ind_K^H(V')$ be a nonzero $H$-stable subspace. 
By~\eqref{eqn:decomp-res-ind-2}, $\res^H_N(V'')$ is a direct sum of representations of the form ${}^g V$ with $g \in H$. Since this representation is a restriction of an $H$-module, it must contain a subrepresentation isomorphic to $V$, hence $V''$ intersects the $V$-isotypic component of $\res^H_N (\ind_K^H(V'))$. Now, again by~\eqref{eqn:decomp-res-ind-2} and the definition of $K$, this isotypic component identifies with the image of the adjunction morphism $V' \to \res^H_K (\ind_K^H(V'))$. By simplicity of $V'$, $V''$ therefore contains this image. Finally, since it is $H$-stable we must have $V''=\ind_K^H(V')$, proving that our map is well defined. To prove that it is injective, one simply notes that if $V'$ and $V''$ are two $K$-representations such that $\res^K_N(V')$ and $\res^K_N(V'')$ are direct sums of copies of $V$, any isomorphism $\ind_K^H(V') \simto \ind_K^H(V'')$ must restrict to an isomorphism between the $V$-isotypic components of $\res^H_N(\ind_K^H(V'))$ and $\res^H_N(\ind_K^H(V''))$, which as explained above identify with $V'$ and $V''$ respectively.

Once this claim is established, we observe that, by~\eqref{eqn:decomp-res-ind}, $\ind_N^K(V)$ is isomorphic, as an $N$-representation, to a direct sum of copies of $V$; hence so does any of its $K$-subrepresentations. Using also exactness and transitivity of induction, this shows that the map of the lemma is well defined and injective. To prove surjectivity, we observe that any simple $H$-subrepresentation $V'$ of $\ind_N^H(V)$ must contain a sub-$N$-representation isomorphic to $V$, see Corollary~\ref{cor:subobjs-quotients}. Then, by Clifford's theorem (see~\cite[Chap.~2, Theorem~2.2(iii)]{karpilovsky2} or~\cite[Theorem~11.1]{Curtis1990}, $V'$ is isomorphic to the image under $\ind_K^H$ of its $V$-isotypic component (as an $N$-representation), and the latter is a simple $K$-representation. This isotypic component must be contained in that of $\ind_N^H(V)$, which identifies with $\ind_N^K(V)$ by~\eqref{eqn:decomp-res-ind}, concluding the proof.

\end{proof}

In the next statement we use the notion of crossed product of a group over an algebra, which generalizes twisted group algebras to the setting where the base field (or ring) can have a nontrivial action of the group; see~\cite[Chap.~3, \S 3]{karpilovsky} for details.

\begin{proposition}
\label{prop:bijection-simple-subreps}

Let $V$ be a simple $N$-module, and set $K=\Stab_H(V)$.

\begin{enumerate}
\item
\label{it:crossed-product}
The algebra $\en_H(\ind_N^H(V))$ is a crossed product of $K/N$ over the division algebra $\en_N(V)$.
\item
\label{it:bijection-simple-subreps}
There exists a canonical bijection between the sets of isomorphism classes of simple subrepresentations of $\ind_N^H(V)$ and of isomorphism classes of simple modules over the algebra $\en_H(\ind_N^H(V))$.
\end{enumerate}

\end{proposition}
	
\begin{proof}
\eqref{it:crossed-product}
As explained in~\cite[Chap.~3, Lemma~12.1]{karpilovsky},
the algebra morphism
\begin{equation}
\label{eqn:isom-End}
\en_K(\ind_N^K(V)) \to \en_H(\ind_N^H(V))
\end{equation}
induced by the functor $\ind_K^H$
is an isomorphism. 
 
Then the claim follows from~\cite[Chap.~3, Theorem~4.2(ii)]{karpilovsky}.

\eqref{it:bijection-simple-subreps}
Using the isomorphism~\eqref{eqn:isom-End} and Lemma~\ref{lem:subreps-stabilizer}, to prove the proposition it suffices to construct a canonical bijection between the sets of isomorphism classes of simple subrepresentations of $\ind_N^K(V)$ and of simple modules over $\en_K(\ind_N^K(V))$.

By~\cite[Theorem~11.17]{Curtis1990}, there exists a functorial isomorphism between the lattices of left ideals in $\en_K(\ind_N^K(V))$ and of $K$-subrepresentations of $\ind_N^K(V)$. This bijection necessarily induces a bijection between the sets of isomorphism classes of simple $\en_K(\ind_N^K(V))$-submodules of $\en_K(\ind_N^K(V))$ and the set of simple $K$-subrepresentations of $\ind_N^K(V)$, so that to conclude it suffices to show that any simple $\en_K(\ind_N^K(V))$-module injects in the free rank-$1$-module. This property follows from the fact that $\en_K(\ind_N^K(V))$ is a symmetric algebra, see~\cite[Chap.~3, Theorem~12.3]{karpilovsky}.
\end{proof}

\begin{remark}
\label{rmk:twisted-gp-alg}
The assertion in Proposition~\ref{prop:bijection-simple-subreps}\eqref{it:crossed-product} means in particular that
the algebra $\en_H(\ind_N^H(V))$ admits a canonical $K/N$-grading, with degree-$1$ component the division algebra $\en_N(V)$. This grading can be described explicitly as follows.
Recall from~\eqref{eqn:isom-End} that we have an algebra isomorphism
\[
\en_H(\ind_N^H(V)) \cong \en_K(\ind_N^K(V));
\]
moreover, by adjunction the right-hand side identifies with $\Hom_N(V,\res^K_N(\ind_N^K(V)))$  (as a $\bk$-vector space). Now, by~\eqref{eqn:decomp-res-ind} $\res^K_N(\ind_N^K(V))$ has a canonical decomposition as a direct sum parametrized by $K/N$, with each factor (noncanonically) isomorphic to $V$. 
We deduce a canonical decomposition of $\en_H(\ind_N^H(V))$ as a direct sum of subspaces parametrized by $K/N$, as desired. (All of these subspaces are noncanonically isomorphic to $\en_N(V)$.)

Assuming that $V$ is absolutely irreducible, so that $\en_N(V)=\bk$, this description simplifies, and reduces to a twisted group algebra for $K/N$ over $\bk$ in the sense of~\cite[Example~8.33(ii)]{Curtis1990} or~\cite[Chap.~3, \S 6]{karpilovsky}.
	\end{remark}

	\section{Definitions}
	\label{sec:definitions}

\subsection{Parabolic and Levi subgroups}

As usual, given a complex algebraic group $H$, we will denote by $H^\circ$ its neutral connected component, and by $\rZ(H)$ its center. Given a subset $K \subset H$, we denote by $\rZ_H(K)=\{h \in H \mid \forall k \in K, hkh^{-1} = k\}$ the centralizer of $K$ in $H$, and by $\rN_H(K)=\{h \in H \mid hKh^{-1} = K\}$ its normalizer.

Let $G$ be a (possibly disconnected) complex reductive algebraic group. Following~\cite[\S 2.3]{Dillery2024}, we say that a subgroup $P$ of $G$ is \emph{parabolic} if the quotient variety $G/P$ is proper, or equivalently if $P^\circ$ is a parabolic subgroup (in the usual sense) of the connected reductive algebraic group $G^\circ$. See~\cite[\S 2.3]{Dillery2024} for a combinatorial description of conjugacy classes of such subgroups. If $P$ is a parabolic subgroup of $G$, with unipotent radical $U$, then a Levi factor of $P$ is a subgroup $L$ such that $P=L \ltimes U$. Given $P$, such a subgroup always exists; see again~\cite[\S 2.3]{Dillery2024} for details. For simplicity, we will use the terminology ``Levi subgroup'' of $G$ for ``Levi factor of a parabolic subgroup'' of $G$.
	

A Levi subgroup $L$ of $G$ will be called \emph{special}\footnote{Our terminology is not the same as that used in~\cite{Aubert2018} and~\cite{Dillery2024}, where these subgroups are called quasi-parabolic subgroups. We find this terminology misleading as special parabolic subgroups are really \emph{special cases} of parabolic subgroups. \label{fn:Levi}} if we have
\[
L=\rZ_G(\rZ(L^{\circ})^{\circ}).
\]
Note that in case $G$ is connected all Levi subgroups are special.
In general, as explained in~\cite[\S 2.6]{Dillery2024}, the assignments $L \mapsto L^\circ$ and $M \mapsto \rZ_G(\rZ(M)^{\circ})$ induce ($G$-equivariant) bijections between the sets of special Levi subgroups of $G$ and Levi subgroups of $G^\circ$. If $P$ is a parabolic subgroup of $G$, we will say that $P$ is special if one (equivalently, all) of its Levi factors is special.

\begin{remark}
\label{rmk:parabolics}
In this paper we will mainly work with the special parabolic subgroups of $G$, because we believe that this is the most natural class of parabolic subgroups from the point of view of the generalized Springer correspondence (in particular, in terms of its inductive nature). This class is also the one emphasized in~\cite{Aubert2018, Dillery2024} because it is the most natural for applications in the Langlands program. It is however noted in~\cite[\S 2.6]{Dillery2024} that many constructions related to the generalized Springer can be stated for other classes of parabolic subgroups. One can easily check that all our considerations below apply if the condition of being special is replaced by the condition of belonging to a class $\mathcal{P}$ of parabolic subgroups of all reductive algebraic groups which satisfies the following conditions:
\begin{enumerate}
\item
$\mathcal{P}$ is admissible in the sense of~\cite[\S 2.6]{Dillery2024};
\item
for any reductive group $G$, the assignment $P \mapsto P^\circ$ induces a bijection between parabolic subgroups of $G$ which belong to $\mathcal{P}$ and parabolic subgroups of $G^\circ$;
\item
for any parabolic subgroup $P \subseteq G$ which belongs to $\mathcal{P}$, if $L$ is a Levi factor of $P$, $L$ is normal in $\rN_G(L^\circ)$.
\end{enumerate}
(Here, the first condition is necessary in the application of the Mackey formula. The second condition is necessary for the definition of cuspidality to be reasonable, see~\cite[Lemma~2.6.6]{Dillery2024}. The third condition is used in the proof of Proposition~\ref{prop:Galois-covering} below.) We can think of only two natural classes of parabolic subgroups which satisfy these conditions, namely special parabolic subgroups and parabolic subgroups of neutral components. Of these two, the first one seems the most natural. But it might be useful to allow more general classes in the theory.
\end{remark}


We will say that a parabolic subgroup $P$ of $G$ is \emph{maximal} if $P^\circ = G^\circ$. By the comments above, there exists exactly one parabolic subgroup which is both special and maximal, but this subgroup is not necessarily $G$ (since the latter is not necessarily a special parabolic subgroup). This subgroup \emph{is} $G$ in case $G$ is semisimple, or if $G$ is itself a special Levi subgroup in a larger reductive algebraic group $H$.

We record the following property for later use.

\begin{lemma}
\label{lem:conn-comp-parabolic}
If $P \subseteq G$ is a parabolic subgroup, the induced morphism $P/P^\circ \to G/G^\circ$ is injective.
\end{lemma}

\begin{proof}
This statement amounts to saying that $P \cap G^\circ = P^\circ$. However $P \cap G^\circ$ is a parabolic subgroup of $G^\circ$, in which $P^\circ$ has finite index, so it must coincide with $P^\circ$.
\end{proof}
	
%

	\subsection{Equivariant sheaves on nilpotent cones}

From now on we fix a field $\bk$ of coefficients for our (complexes of) sheaves, and continue with our possibly disconnected complex reductive algebraic group $G$.
For a complex algebraic group $H$, we will denote by $\mathfrak{h}$ its Lie algebra, and by $\nn_H$ the nilpotent cone of $H$, i.e.~the closed subvariety of $\mathfrak{h}$ consisting of nilpotent elements. Note that we have $\nn_H=\nn_{H^\circ}$. For $x \in \mathfrak{h}$ we will denote by $\rZ_H(x)$ the centralizer of $x$ in $H$, and set $\rA_H(x) := \rZ_H(x)/\rZ_H(x)^\circ$.

The main player in this paper will be the derived category $\Db_G(\nn_G,\bk)$ of constructible sheaves on the quotient stack $G \backslash \nn_G$, with coefficients in $\bk$. This category can also be described in terms of the Bernstein--Lunts construction, see~\cite[Chap.~6]{Achar2021}. We have a canonical ``forgetful'' functor
\[
\mathrm{For}^G : \Db_G(\nn_G,\bk) \to \Db_{\mathrm{c}}(\nn_G,\bk)
\]
where the right-hand side is the constructible derived category of sheaves on $\nn_G$. The category $\Db_G(\nn_G,\bk)$ is equipped with a natural t-structure such that $\mathrm{For}^G$ is t-exact with respect to the perverse t-structure, see e.g.~\cite[Theorem~6.4.10]{Achar2021}. This t-structure is also called the perverse t-structure, and its heart will be denoted $\Perv_G(\nn_G,\bk)$. In fact, by~\cite[Theorem~6.4.10]{Achar2021} and its proof, this heart consists precisely of the objects of $\Db_G(\nn_G,\bk)$ whose image in $\Db_{\mathrm{c}}(\nn_G,\bk)$ is perverse, and it can equivalently be described as the category of $G$-equivariant perverse sheaves on $\nn_G$ in the sense of~\cite[Definition~6.2.3]{Achar2021}. In case $G$ is connected the restriction of the functor $\mathrm{For}^G$ to $\Perv_G(\nn_G,\bk)$ is fully faithful (so that a $G$-equivariant perverse sheaf is a perverse sheaf with a certain \emph{property}), and its essential image is stable under subquotients, but in general these properties are not true (in particular, a $G$-equivariant perverse sheaf should be considered as a perverse sheaf with a certain additional \emph{structure}).

As in the introduction, we will denote by $\fR_{G,\bk}$ the set of isomorphism classes of simple objects of $\Perv_G(\nn_G,\bk)$. By the general theory of perverse sheaves, every object in $\Perv_G(\nn_G,\bk)$ has finite length, and $\fR_{G,\bk}$ identifies with the quotient of the set of pairs $(\cO,\cE)$ where $\cO \subset \nn_G$ is a $G$-orbit and $\cE$ is a simple $G$-equivariant $\bk$-local system on $\cO$, by the equivalence relation given by $(\cO,\cE) \sim (\cO',\cE')$ if $\cO=\cO'$ and $\cE \cong \cE'$. This bijection sends the equivalence class of a pair $(\cO,\cE)$ to the isomorphism class of the intersection cohomology complex $\IC(\cO,\cE)$. We will freely identify these two sets in the rest of the paper. Recall also that given a $G$-orbit $\cO \subset \nn_G$ and a point $x \in \cO$, restriction to $x$ induces an equivalence between the category of $G$-equivariant $\bk$-local systems on $\cO$ and $\Rep_\bk(\rA_G(x))$, see~\cite[Proposition~6.2.13]{Achar2021}.

Let us now study the relation between the categories $\Db_G(\nn_G,\bk)$ and $\Db_{G^\circ}(\nn_{G^\circ},\bk)$.
As explained above $G$ and $G^\circ$ share the same nilpotent cone, so that we have a canonical t-exact forgetful functor
\[
\mathrm{For}^G_{G^\circ} : \Db_G(\nn_G,\bk) \to \Db_{G^\circ}(\nn_{G},\bk).
\]
In the language of stacks, this functor is the ($*$- or $!$-, equivalently) pullback under the canonical morphism $G^\circ \backslash \nn_G \to G \backslash \nn_G$. It admits a left and right adjoint
\[
\gamma_{G^\circ}^G : \Db_{G^\circ}(\nn_{G},\bk) \to \Db_{G}(\nn_{G},\bk),
\]
given by ($*$- or $!$-, equivalently) pushforward under this map.


\begin{lemma}
\phantomsection
\label{lem:ind-res-ss}
\begin{enumerate}
\item
\label{it:rest-ss}
The functor $\mathrm{For}^G_{G^\circ} : \Perv_G(\nn_G,\bk) \to \Perv_{G^\circ}(\nn_{G^\circ},\bk)$ sends semisimple objects to semisimple objects.
\item
\label{it:ind-exact-ss}
The functor $\gamma_{G^\circ}^G : \Db_{G^\circ}(\nn_{G^\circ},\bk) \to \Db_{G}(\nn_{G},\bk)$ is t-exact.
If $\mathrm{char}(\bk)$ does not divide $|G/G^\circ|$, this functor sends semisimple perverse sheaves to semisimple perverse sheaves.
\end{enumerate}
\end{lemma}

\begin{proof}
\eqref{it:rest-ss}
It suffices to prove that for any $G$-orbit $\cO \subset \nn_G$ and any simple $G$-equivariant $\bk$-local system $\cE$ on $\cO$, the perverse sheaf $\mathrm{For}^G_{G^\circ}(\IC(\cO,\cE))$ is semisimple. Consider the decomposition $\cO = \cO_1 \sqcup \cdots \sqcup \cO_r$ of $\cO$ into $G^\circ$-orbits, and choose for any $i$ a point $x_i \in \cO_i$. Then we have
\begin{equation}
\label{eqn:For-IC}
\mathrm{For}^G_{G^\circ}(\IC(\cO,\cE)) = \bigoplus_{i=1}^r \IC(\cO_i,\cE_{|\cO_i}).
\end{equation}
Moreover, if $i \in \{1, \dots, r\}$,
denoting by $V_i$ the representation of $\rA_G(x_i)$ corresponding to the $G$-equivariant local system $\cE$ on $\cO = G \cdot x_i$, the restriction $\cE_{|\cO_i}$ is the $G^\circ$-equivariant local system on $\cO_i = G^\circ \cdot x_i$ corresponding to the restriction of $V_i$ to the normal subgroup $\rA_{G^\circ}(x_i) $. (To make sense of this we use the fact that $\rZ_{G}(x_i)^\circ = \rZ_{G^\circ}(x_i)^\circ$.) Hence the desired claim follows from Lemma~\ref{lem:ind-res-ss-gps}\eqref{it:res-ss-gps}

\eqref{it:ind-exact-ss}
It suffices to prove that for any $G^\circ$-orbit $\cO \subset \nn_G$ and any simple $G^\circ$-equivariant $\bk$-local system $\cE$ on $\cO$, the complex $\gamma_{G^\circ}^G(\IC(\cO,\cE))$ is perverse, and semisimple when $\mathrm{char}(\bk)$ does not divide $|G/G^\circ|$. Fix such data, and choose $x \in \cO$. Concretely, setting $\cO':=G \cdot x$, this complex is obtained by taking pushforward under the canonical morphism $f : G \times^{G^\circ} \overline{\cO} \to \overline{\cO'}$ of the intersection cohomology complex $\IC(G \times^{G^\circ} \cO, \cE')$ where $\cE'$ is the $G$-equivariant local system on $G \times^{G^\circ} \cO$ corresponding to $\cE$ under the canonical equivalence between $G$-equivariant local systems on $G \times^{G^\circ} \cO$ and $G^\circ$-equivariant local systems on $\cO$. Now this map is finite, and its restriction to $G \times^{G^\circ} \cO = f^{-1}(\cO')$ is a Galois covering with Galois group $\rZ_{G}(x)/\rZ_{G^\circ}(x)$. Hence, as e.g.~in~\cite[Proposition~3.8.10]{Achar2021} $f_* \IC(G \times^{G^\circ} \cO, \cE')$ is the intersection cohomology complex of the local system on $\cO'$ obtained from $\cE'$ by pushforward along the covering $G \times^{G^\circ} \cO \to \cO'$. In particular, this complex is perverse. In terms of representations of finite groups, this pushforward corresponds to the representation of $\rA_G(x)$ obtained from the representation of $\rA_{G^\circ}(x)$ associated with $\cE$ by induction. Hence the semisimplicity claim follows from Lemma~\ref{lem:ind-res-ss-gps}\eqref{it:ind-ss-gps}.
(Here $\rA_G(x) / \rA_{G^\circ}(x)$ is a subgroup of $G/G^\circ$, hence its order is prime to $\mathrm{char}(\bk)$.)
\end{proof}

\begin{remark}
Of course, the assumption in the second statement of Lemma~\ref{lem:ind-res-ss}\eqref{it:ind-exact-ss} is necessary, as can be seen already when $G$ is a finite group.
\end{remark}

	\subsection{Induction and restriction functors}
	
Let $P$ be a parabolic subgroup of $G$, and $L \subset P$ be a Levi factor. Then $L$ is also naturally a quotient of $P$, and we consider the associated correspondence of stacks
\begin{equation}
\label{eqn:correspondence-ind-res}
L \backslash \nn_L \xleftarrow{q} P \backslash \nn_P \xrightarrow{i} G \backslash \nn_G.
\end{equation}
We define the associated induction and restriction functors as follows:
\begin{align}
	&\ind_{L\subseteq P}^G:=i_! q^* : \Db_L(\nn_L,\bk)\to \Db_G(\nn_G,\bk), \\
	&\res_{L\subseteq P}^G:= q_! i^*: \Db_G(\nn_G,\bk) \to \Db_L(\nn_L,\bk), \\
	&\pres_{L\subseteq P}^G:= q_* i^!: \Db_G(\nn_G,\bk) \to \Db_L(\nn_L,\bk).
	\end{align}

\begin{remark}
\label{rmk:ind-res}
In practice we will mainly use these functors in case $L$ is special, but it can sometimes be convenient to consider them for more general parabolic subgroups. For instance, for $P=G^\circ$ we have $\res_{G^\circ \subseteq G^\circ}^G=\pres_{G^\circ \subseteq G^\circ}^G = \mathrm{For}^G_{G^\circ}$ and $\ind_{G^\circ \subseteq G^\circ}^G=\gamma_{G^\circ}^G$.
\end{remark}
	
	
\begin{proposition}
\phantomsection
\label{indprop}
	\begin{enumerate}
		\item 
\label{it:adjunction}		
		The functor $\res_{L\subseteq P}^G$, resp.~$'\res_{L\subseteq P}^G$, is left, resp.~right, adjoint to $\ind_{L\subseteq P}^G$.
		\item 
\label{it:transitivity}	
		The formation of these functors is transitive, in the sense that if $P\subseteq Q\subseteq G$ are parabolic subgroups with respective Levi factors $L\subseteq M$ then we have
\[	
	\ind_{M\subseteq Q}^G\circ \ind_{L\subseteq M\cap P}^M=\ind_{L\subseteq P}^G, \ \res_{L\subseteq M\cap P}^M \circ \res_{M\subseteq Q}^G = \res_{L\subseteq P}^G, \ \pres_{L\subseteq M\cap P}^M \circ \pres_{M\subseteq Q}^G = \pres_{L\subseteq P}^G.
\]
		\item 
\label{it:exactness}		
		The functors $\ind_{L\subseteq P}^G$, $\res_{L\subseteq P}^G$ and $\pres_{L\subseteq P}^G$ are $t$-exact with respect to the perverse $t$-structures.
		\end{enumerate}
\end{proposition}

\begin{proof}
The proofs are similar to those in the connected case (see~\cite[\S 2.1]{Achar2016a}) or for characteristic-$0$ coefficients (see~\cite{Dillery2024}). Namely,~\eqref{it:adjunction} follows from the facts that the map $\nn_P \to \nn_G$ is a closed immersion, while the map $\nn_P \to \nn_L$ is smooth. \eqref{it:transitivity} follows from an application of the base change theorem. Finally, for~\eqref{it:exactness}, as in~\cite[Corollary~2.4.4]{Dillery2024} we have
\begin{equation}
\label{eqn:Res-For}
\mathrm{For}^L_{L^\circ} \circ \res_{L\subseteq P}^G \cong \res^{G^\circ}_{L^\circ \subseteq P^\circ} \circ \mathrm{For}^G_{G^\circ}.
\end{equation}
Therefore,
t-exactness of $\res_{L\subseteq P}^G$ follows from that of $\res^{G^\circ}_{L^\circ \subseteq P^\circ}$, for which we refer to~\cite[\S 2.1]{Achar2016a}.
One checks similarly that 
$\pres_{L\subseteq P}^G$ is t-exact. Then $\ind_{L\subseteq P}^G$ is t-exact because it admits t-exact left and right adjoints.
\end{proof}
	
\begin{remark}
When $\mathrm{char}(\bk)=0$ the functors $\ind_{L\subseteq P}^G$, $\res_{L\subseteq P}^G$ and $\pres_{L\subseteq P}^G$ preserves semisimple complexes. (This follows from the Decomposition Theorem, see~\cite[Lemma~2.4.6]{Dillery2024}.) In general this property is not true, already for connected groups.
For instance, when $G=\mathrm{SL}_2(\mathbb{C})$ with $B$ the Borel subgroup of upper triangular matrices and $T$ the maximal torus of diagonal matrices, then for the skyscraper sheaf $\delta_0$ at the origin of $\nn_T=\{0\}$ the Springer sheaf $\ind_{T\subseteq B}^G (\delta_0)$ is not semisimple if $\mathrm{char}(\bk)=2$, see e.g.~\cite[\S 2.4]{Juteau2012}.
\end{remark}

\subsection{Cuspidal local systems and cuspidal triples}
\label{ss:cuspidal}

	
\begin{definition}
\label{def:cuspidal}
A perverse sheaf $\f \in \Perv_G(\nn_G,\bk)$ is called \emph{cuspidal} if $\res_{L\subseteq P}^G(\f)=0$ for every non-maximal special parabolic subgroup $P$ with Levi factor $L$.
		We call a pair $(\cO, \cE)$ of a $G$-orbit $\cO \subset \nn_G$ and a simple $G$-equivariant $\bk$-local system $\cE$ on $\cO$ a \emph{cuspidal pair} if the associated simple equivariant perverse sheaf $\IC(\cO,\cE)$ is cuspidal.
	\end{definition}
	
We will denote by
\[
\fR_{G,\bk}^{\cusp} \subset \fR_{G,\bk}
\]
the subset of (isomorphism classes of) cuspidal objects.

If $\f \in \fR_{G,\bk}$, using~\eqref{eqn:Res-For} and the fact that the functors $\mathrm{For}^G_{G^\circ}$ do not kill any nonzero object one sees that $\f$ is cuspidal if and only if each composition factor of $\mathrm{For}^G_{G^\circ}(\f)$ is cuspidal in the sense above for the group $G^\circ$, or in other words in the sense of~\cite{Achar2016a}. (In fact, in view of Lemma~\ref{lem:ind-res-ss}\eqref{it:rest-ss}, one can replace ``composition factor'' in this sentence by ``simple direct summand.'')
Using a similar property for the functors $\pres_{L\subseteq P}^G$, and~\cite[Proposition~2.1]{Achar2016a}, one sees that $\f$ is cuspidal if and only if it is annihilated by $\pres_{L\subseteq P}^G$ for any non-maximal special parabolic subgroup $P$ with Levi factor $L$. Then, arguing as in~\cite[Proposition~2.1]{Achar2016a}, one sees that this condition is equivalent to the following ones:
\begin{itemize}
\item
$\f$ is not isomorphic to a quotient of an object of the form $\ind_{L \subseteq P}^G(\cG)$ with $P$ a non-maximal special parabolic subgroup $P$ with Levi factor $L$ and $\cG \in \Perv_L(\nn_L,\bk)$;
\item
$\f$ is not isomorphic to a subobject of an object of the form $\ind_{L \subseteq P}^G(\cG)$ with $P$ a non-maximal special parabolic subgroup $P$ with Levi factor $L$ and $\cG \in \Perv_L(\nn_L,\bk)$.
\end{itemize}

\begin{remark}
\phantomsection
\label{rmk:cuspidal}
\begin{enumerate}
\item
\label{it:comparsion-cuspidal-AMS}
In case $\mathrm{char}(\bk)=0$, our definition corresponds to the one given in~\cite[\S 2.6]{Dillery2024}, for the admissible class $\mathcal{P}$ consisting of special parabolic subgroups. In case $G$ is connected, this definition coincides with that given in~\cite{Achar2016a}. See~\cite[\S 2.2]{Achar2016a} for comments on the comparison with the definition used by Lusztig (when $\mathrm{char}(\bk)=0$ and $G$ is connected) in~\cite{Lusztig1984}. 

Let us now compare our definition with that chosen in~\cite{Aubert2018}.
Let $(\cO,\cE)$ be a pair in $\fR_{G,\bk}$, and consider the decomposition $\cO = \cO_1 \sqcup \cdots \sqcup \cO_r$ of $\cO$ into $G^\circ$-orbits. Recall the proof of Lemma~\ref{lem:ind-res-ss}\eqref{it:rest-ss}, and in particular the decomposition~\eqref{eqn:For-IC}. The factors in this decomposition are conjugate to each other under automorphisms of $\nn_G$ induced by automorphisms of $G^\circ$, hence one is cuspidal if and only if all of them are cuspidal. Choosing $x \in \cO$, and denoting by $i$ the index such that $x \in \cO_i$, from this proof one sees that $(\cO,\cE)$ is cuspidal if and only if all the pairs $(\cO_i,\cE')$ where $\cE'$ runs over the composition factors of the (semisimple) $G^\circ$-equivariant local system $\cE_{|\cO_i}$ are cuspidal. 
In particular, in the setting of~\cite{Aubert2018} (corresponding for us to the case $\mathrm{char}(\bk)=0$), Definition~\ref{def:cuspidal} is equivalent to that given in~\cite[\S 3]{Aubert2018}, up to the identification of the unipotent variety and the nilpotent cone, and of representations of component groups of stabilizers with equivariant local systems on orbits.
\item
It is asserted in~\cite[\S 3]{Aubert2018}, in the setting where $\mathrm{char}(\bk)=0$, that if $(\cO,\cE)$ belongs to $\fR_{G,\bk}^{\cusp}$ then $\cO$ is a single $G^\circ$-orbit. This fact is however false. In fact, consider a connected reductive algebraic group $H$ such that there exist two pairs $(\cO_1,\cE_1)$ and $(\cO_2,\cE_2)$ in $\fR_{H,\bk}^{\cusp}$ with $\cO_1 \neq \cO_2$. (This situation can occur in particular for Spin groups.) Then, setting $G = (H \times H) \rtimes (\bZ/2\bZ)$ with $\bZ/2\bZ$ permuting the two factors, $\gamma_{G^\circ}^G(\IC(\cO_1,\cE_1) \boxtimes \IC(\cO_2,\cO_2))$ is a simple $G$-equivariant perverse sheaf which, as $G^\circ$-equivariant perverse sheaf, is isomorphic to $\IC(\cO_1,\cE_1) \boxtimes \IC(\cO_2,\cE_2) \oplus \IC(\cO_2,\cE_2) \boxtimes \IC(\cO_1,\cE_1)$. In particular it is cuspidal, but the corresponding $G$-orbit, namely $(\cO_1 \times \cO_2) \sqcup (\cO_2 \times \cO_1)$, is not a single $G^\circ$-orbit.
\end{enumerate}
\end{remark}
	
\begin{definition}
A \emph{cuspidal triple} for $G$ is a triple $(L,\cO,\cE)$ where $L$ is a special Levi subgroup of $G$, $\cO \subset \nn_{L}$ is an $L$-orbit, and $\cE$ is a simple $L$-equivariant $\bk$-local system on $\cO$ such that $(\cO,\cE)$ is a cuspidal pair for $L$.
\end{definition}

We introduce an equivalence relation on the set of cuspidal triples as follows. If $(L,\cO,\cE)$ and $(L',\cO',\cE')$ are cuspidal triples, we write $(L,\cO,\cE) \sim (L',\cO',\cE')$ if there exists $g \in G$ such that $L'=gLg^{-1}$, $\cO' = g \cdot \cO$, and $\cE'$ is isomorphic to the image of $\cE$ under the pushforward functor associated with the isomorphism $\cO \simto \cO'$ induced by $g$. We will denote by $\fL$ the quotient of the set of cuspidal triples by this equivalence relation.
	

\subsection{Fourier transform and applications}
\label{ss:Fourier}

The next tool we will need in our study is Fourier transform. This subsection is parallel to~\cite[\S 2.4]{Achar2016a} or~\cite[\S 3.4]{Dillery2024}.

Here we work as in~\cite[\S 2.4]{Achar2016a} with Fourier--Sato transform. (One can also work with the Fourier--Laumon transform as in~\cite[\S 3.4]{Dillery2024}, since we will only use some formal properties that the two constructions enjoy.) First, we note that there exists a $G$-invariant non-degenerate symmetric bilinear form on $\mathfrak{g}$. In fact,
we have a $G$-equivariant decomposition $\mathfrak{g} = [\mathfrak{g},\mathfrak{g}] \oplus \mathfrak{z}(\mathfrak{g})$ where $\mathfrak{z}(\mathfrak{g})$ is the center of $\mathfrak{g}$.
If we choose $G$-invariant non-degenerate symmetric bilinear forms on $[\mathfrak{g},\mathfrak{g}]$ and $\mathfrak{z}(\mathfrak{g})$, we can assemble them to obtain the desired form on $\mathfrak{g}$. Here, on $[\mathfrak{g},\mathfrak{g}]$ one can e.g.~take the Killing form. To show existence of a $G$-invariant nondegenerate symmetric bilinear form on $\mathfrak{z}(\mathfrak{g})$ one can proceed as follows. This Lie algebra is the Lie algebra of the torus $\rZ(G^\circ)^\circ$. Hence we have $\mathfrak{z}(\mathfrak{g}) = \mathbb{C} \otimes_{\mathbb{Z}} \mathrm{X}_*(\rZ(G^\circ)^\circ)$. The action of $G$, which factors through an action of the finite group $G/G^\circ$, is induced by an action on the lattice $\mathrm{X}_*(\rZ(G^\circ)^\circ)$; it therefore admits a real form, hence an invariant non-degenerate bilinear form.

After fixing such a form, Fourier transform provides a t-exact autoequivalence $\mathbb{T}_G$ of the derived category of $G$-equivariant conical sheaves on $\mathfrak{g}$. (The meaning of ``conical'' will not be important here. All that we will use is that these objects form a full subcategory of the $G$-equivariant derived category of $\mathfrak{g}$, and that the pushforward to $\mathfrak{g}$ of any $G$-equivariant perverse sheaf on $\nn_G$ is conical.)

\begin{lemma}
\label{lem:Fourier-cusp}
For any $(\cO,\cE) \in \fR_{G,\bk}^{\cusp}$, there exists a pair $(\cO',\cE') \in \fR_{G,\bk}^{\cusp}$ such that
\[
\mathbb{T}_G(\IC(\cO,\cE)) \cong \IC(\cO' + \mathfrak{z}(\mathfrak{g}), \cE' \boxtimes \underline{\bk}_{\mathfrak{z}(\mathfrak{g})}).
\]
\end{lemma}

\begin{proof}
The proof is similar to that in the connected case, see~\cite[Proposition~2.11 and Corollary~2.12]{Achar2016a}. Namely, we first assume that $G^\circ$ is semisimple. In this case the claim amounts to saying that the simple perverse sheaf $\mathbb{T}_G(\IC(\cO,\cE))$ is supported on $\nn_G$, and cuspidal. As explained in~\S\ref{ss:cuspidal} these properties can be checked after applying the forgetful functor $\mathrm{For}^G_{G^\circ}$. Since $\mathrm{For}^G_{G^\circ}(\IC(\cO,\cE))$ is cuspidal and $\mathrm{For}^G_{G^\circ}(\mathbb{T}_G(\IC(\cO,\cE))) \cong \mathbb{T}_{G^\circ}(\mathrm{For}^G_{G^\circ}(\IC(\cO,\cE)))$, they follow from the connected case treated in~\cite[Proposition~2.11]{Achar2016a}.

The general case follows by comparing with the case of $G/\rZ(G^\circ)^\circ$.
\end{proof}

\begin{remark}
As in~\cite[Remark~2.13]{Achar2016a} or~\cite[Conjecture~3.4.2]{Dillery2024} it seems reasonable to expect that in the setting of Lemma~\ref{lem:Fourier-cusp} we always have $(\cO',\cE')=(\cO,\cE)$. However this fact is already not known if $G$ is connected and $\mathrm{char}(\bk)>0$, or when $G$ is disconnected and $\mathrm{char}(\bk)=0$.
\end{remark}

Note that for any Levi subgroup $L$ of $G$, the restriction of our non-degenerate $G$-invariant bilinear form on $\mathfrak{g}$ to $\mathfrak{l}$ is a non-degenerate $L$-invariant bilinear form on $\mathfrak{l}$, so that it makes sense to consider the Fourier transform $\mathbb{T}_L$.

We now fix a special Levi subgroup $L$ of $G$, an $L$-orbit $\cO \subseteq \nn_L$, and an $L$-equivariant local system $\cE$ on $\cO$. We set
\[
\mathfrak{z}(\mathfrak{l})^\circ := \{ x \in \mathfrak{z}(\mathfrak{l}) \mid \rZ_G(x)^\circ = L^\circ\}.
\]
(Note that here we have $\rZ_G(x)^\circ = \rZ_{G^\circ}(x)^\circ$, so that this definition is in a sense insensitive to disconnectedness.) Let also
\[
\widetilde{Y}_{(L,\cO)} := G \times^L (\cO + \mathfrak{z}(\mathfrak{l})^\circ), \quad Y_{(L,\cO)} := G \cdot (\cO + \mathfrak{z}(\mathfrak{l})^\circ),
\]
and consider the stabilizer $\rN_G(L,\cO)$ of $\cO$ in $\rN_G(L)$.
We have a natural map
\[
\pi_{(L,\cO)}: \widetilde{Y}_{(L,\cO)}\to Y_{(L,\cO)}
\]
induced by the $G$-action on $\mathfrak{g}$.
	
\begin{proposition}
\label{prop:Galois-covering}
	The morphism $\pi_{(L,\cO)}$ is a Galois covering with Galois group $\rN_G(L,\cO)/L$.
\end{proposition}

\begin{proof}
The statement will follow from~\cite[Lemma 3.3.2]{Dillery2024}, provided we prove that the image of $L$ in $\rN_G(L^\circ,\cO) / \rN_G(L^\circ,\cO)^\circ$ is normal in our setting. (Here, $\rN_G(L^\circ,\cO)$ is the stabilizer of $\cO$ in $\rN_G(L^\circ)$.) 
However, since $L$ is special we have $\rN_G(L)=\rN_G(L^{\circ})$, therefore $\rN_G(L,\cO)=\rN_G(L^{\circ},\cO)$, and thus $L$ is normal in $\rN_G(L^{\circ},\cO)$.
\end{proof}

Since $\widetilde{Y}_{(L,\cO)}$ is defined using induction of varieties, there exists an equivalence of categories between the categories of $G$-equivariant local systems on $\widetilde{Y}_{(L,\cO)}$ and of $L$-equivariant local systems on $\cO + \mathfrak{z}(\mathfrak{l})^\circ$. Recall that we have fixed an $L$-equivariant local system $\cE$ on $\cO$; we will denote by $\widetilde{\cE}$ the $G$-equivariant local system on $\widetilde{Y}_{(L,\cO)}$ corresponding under this equivalence to $\cE \boxtimes \underline{\bk}_{\mathfrak{z}(\mathfrak{l})^\circ}$. By Proposition~\ref{prop:Galois-covering}, using pushforward we obtain the $G$-equivariant local system $(\pi_{(L,\cO)})_* \widetilde{\cE}$ on $Y_{(L,\cO)}$.

Let us now fix a parabolic subgroup $P \subset G$ having $L$ as Levi factor.
As in~\cite[\S 2.4]{Achar2016a} or~\cite[\S 2.4]{Dillery2024}, the induction functor $\ind_{L \subseteq P}^G$ ``extends'' to a functor from $\Db_L(\mathfrak{l},\bk)$ to $\Db_G(\mathfrak{g},\bk)$, which will be denoted similarly, and which sends conical complexes to conical complexes. (For this construction, one simply replaces the nilpotent cones by the full Lie algebras in the correspondence~\eqref{eqn:correspondence-ind-res}.)

\begin{proposition}
\label{prop:induction-IC}
We have a canonical isomorphism
\[
\ind_{L \subseteq P}^G(\IC(\cO + \mathfrak{z}(\mathfrak{l}), \cE \boxtimes \underline{\bk}_{\mathfrak{z}(\mathfrak{l})})) \cong \IC(Y_{(L,\cO)}, (\pi_{(L,\cO)})_* \widetilde{\cE}).
\]
\end{proposition}

\begin{proof}
As in the proof of~\cite[Proposition~2.17]{Achar2016a}, what we have to prove is that the left-hand side satisfies the usual defining conditions for the right-hand side. For this we use the ``more concrete'' description of induction as in~\cite[\S 2.5]{Achar2016a}. From this description, it is clear that this complex is supported on the closure of $Y_{(L,\cO)}$, and one easily checks that its restriction to $Y_{(L,\cO)}$ is $(\pi_{(L,\cO)})_* \widetilde{\cE}$. Next one needs to check that the restriction, resp.~corestriction, of $\ind_{L \subseteq P}^G(\IC(\cO + \mathfrak{z}(\mathfrak{l}), \cE \boxtimes \underline{\bk}_{\mathfrak{z}(\mathfrak{l})}))$ to $\overline{Y_{(L,\cO)}} \smallsetminus Y_{(L,\cO)}$ is concentrated in perverse degrees $\leq -1$, resp.~$\geq 1$. These properties can be checked after applying the functor $\mathrm{For}^G_{G^\circ}$. Now if $\Sigma \subset G$ is a subset whose image in $G/G^\circ$ form cosets representatives for $(G/G^\circ)/(P/P^\circ)$ (see Lemma~\ref{lem:conn-comp-parabolic}) one sees that we have
\[
\mathrm{For}^G_{G^\circ} \bigl( \ind_{L \subseteq P}^G(\IC(\cO + \mathfrak{z}(\mathfrak{l}), \cE \boxtimes \underline{\bk}_{\mathfrak{z}(\mathfrak{l})})) \bigr) \cong \bigoplus_{\sigma \in \Sigma} (\mathrm{Ad}_\sigma)_* \ind_{L^\circ \subseteq P^\circ}^{G^\circ}(\IC(\cO + \mathfrak{z}(\mathfrak{l}), \cE \boxtimes \underline{\bk}_{\mathfrak{z}(\mathfrak{l})})).
\]
In the right-hand side one can write $\cO$ as a union of $G^\circ$-orbits to further decompose this object as a direct sum. Then the direct summands are exactly of the type appearing in~\cite[Proposition~2.17]{Achar2016a}, and one can use this result to conclude.
\end{proof}

\begin{remark}
In the proof of Proposition~\ref{prop:induction-IC}, one can replace the argument using coset representatives for $(G/G^\circ)/(P/P^\circ)$ by a use of the Mackey formula proved in Theorem~\ref{mackey} below.
\end{remark}

One can then derive from these results the following important consequence.
	
\begin{corollary}
\label{cor:subquo}
For a cuspidal triple $(L,\cO,\cE)$, up to isomorphism the perverse sheaf $\ind_{L\subseteq P}^G(\IC(\cO,\cE))$ does not depend on the choice of a parabolic subgroup $P$ of $G$ with Levi factor $L$, and is invariant under the equivalence relation introduced in~\S\ref{ss:cuspidal}.
Moreover, the sets of isomorphism classes of simple subobjects and of simple quotients of $\ind_{L\subseteq P}^G(\IC(\cO,\cE))$ coincide.
\end{corollary}

\begin{proof}
It is enough to prove similar properties for the perverse sheaf $\mathbb{T}_G(\ind_{L\subseteq P}^G(\IC(\cO,\cE)))$. However, as in~\cite[Lemma~2.9]{Achar2016a} induction commutes with Fourier transform in the appropriate sense, so that we have
\[
\mathbb{T}_G(\ind_{L\subseteq P}^G(\IC(\cO,\cE))) \cong \ind_{L\subseteq P}^G(\mathbb{T}_L(\IC(\cO,\cE))).
\]
Now, by Lemma~\ref{lem:Fourier-cusp} we have 
\[
\mathbb{T}_L(\IC(\cO,\cE)) \cong \IC(\cO' + \mathfrak{z}(\mathfrak{l}), \cE' \boxtimes \underline{\bk}_{\mathfrak{z}(\mathfrak{l})})
\]
for some pair $(\cO',\cE') \in \fR_{L,\bk}$. Then we use Proposition~\ref{prop:induction-IC} to deduce an isomorphism
\[
\mathbb{T}_G(\ind_{L\subseteq P}^G(\IC(\cO,\cE))) \cong \IC(Y_{(L,\cO')}, (\pi_{(L,\cO')})_* \widetilde{\cE'}).
\]
The right-hand side clearly does not depend on the choice of $P$, and is invariant under our equivalence relation, which proves the first claim.

The proof of the second claim is identical to that of~\cite[Lemma 2.3]{Achar2017}. We recall its main ingredients (presented from a slightly different angle, which follows~\cite[Remark~3.3]{Achar2017}) for later reference. Once again it is enough to prove the similar claim for $\mathbb{T}_G(\ind_{L\subseteq P}^G(\IC(\cO,\cE)))$. Next, since the $\IC$ functor respects socles and heads, it is enough to prove that the isomorphism classes of simple subobjects and quotients of $(\pi_{(L,\cO')})_* \widetilde{\cE'}$ coincide. For this we fix $x \in \cO'$, and interpret local systems in terms of representations of finite groups. The local system $\cE'$ corresponds to a representation $V$ of $\rA_L(x)$. Now we have
\[
\rA_{\rN_G(L)}(x) / \rA_L(x) \cong \rZ_{\rN_G(L)}(x) / \rZ_L(x) \cong \rN_G(L,\cO')/L,
\]
and this group is the Galois group of the covering $\pi_{(L,\cO')}$ by Proposition~\ref{prop:Galois-covering}. Hence any representation of $\rA_{\rN_G(L)}(x)$ determines a $G$-equivariant local system on $Y_{(L,\cO')}$, and $(\pi_{(L,\cO')})_* \widetilde{\cE'}$ corresponds to the representation $\mathrm{Ind}_{\rA_L(x)}^{\rA_{\rN_G(L)}(x)}(V)$. Hence the claim follows from Corollary~\ref{cor:subobjs-quotients}.
\end{proof}	

\subsection{Induction series}

Thanks to Corollary~\ref{cor:subquo}, it makes sense to define, for a cuspidal triple $(L,\cO,\cE)$, the subset
\[
\fR_{G,\bk}^{(L,\cO,\cE)} \subseteq \fR_{G,\bk}
\]
as the set of isomorphism classes of simple subobjects (equivalently, quotients) of the $G$-equivariant perverse sheaf $\ind_{L\subseteq P}^G(\IC(\cO,\cE))$, for any choice of a parabolic subgroup $P \subset G$ with Levi factor $L$. Moreover, this subset only depends on the image of $(L,\cO,\cE)$ in $\fL$. It will be called the \emph{induction series} associated with (the class of) $(L,\cO,\cE)$.
	
%

	\section{Partition into induction series}
	\label{sec:partition}

The goal of this section is to prove that induction series form a partition of $\fR_{G,\bk}$, see Theorem~\ref{disj}.

\subsection{A Mackey formula}

We first prove a ``Mackey formula'' for compositions of restriction and induction functors, for which the geometric input is given in~\cite{Dillery2024}.
	
For a subgroup $A\subseteq G$ and an element $g\in G$, we denote by ${}^gA$ the subgroup $gAg^{-1}$. We fix parabolic subgroups $P,Q$ of $G$ with respective Levi subgroups $L,M$. By~\cite[Lemma~2.3.1]{Dillery2024}, there exists a sequence $g_1, \dots, g_s$ of representatives for the double cosets in $Q \backslash G / P$ such that:
\begin{itemize}
\item
for any $i$, $M \cap {{}^{g_i} L}$ contains a maximal torus of $G$;
\item
if $Q g_i P \subseteq \overline{Q g_j P}$, then $i \leq j$.
\end{itemize}
Then, for any $i$, $M \cap {}^{g_i} P$ is a parabolic subgroup of $M$ with Levi subgroup $M \cap {}^{g_i} L$, and $Q \cap {}^{g_i} L$ is a parabolic subgroup in ${}^{g_i} L$ with Levi subgroup $M \cap {}^{g_i} L$.

	
\begin{theorem}[Mackey formula]
\label{mackey}
		 Let $\f\in \Perv_L(\nn_L,\bk).$ Then in $\Perv_M(\nn_M,\bk)$ we have a filtration
\[
0=\f_0\subseteq \f_1\subseteq \cdots \subseteq \f_s=\res_{M\subseteq Q}^G\circ \ind_{L\subseteq P}^G(\f)
\]
such that for any $i \geq 1$ we have
\[
\f_i/\f_{i-1}\cong \ind_{M\cap {{}^{g_i}L} \subseteq M\cap {}^{g_i}P}^M\left(\res_{M\cap {}^{g_i}L\subseteq Q\cap {}^{g_i}L}^{^{g_i}L} \mathrm{Ad}(g_i^{-1})^*\f)\right).
\]
\end{theorem}
	
\begin{proof}
	The proof is similar to that of~\cite[Theorem 2.5.1]{Dillery2024}. 
Set $\fX =  P \backslash \cN_P \times_{G \backslash \cN_G} Q \backslash \cN_Q$, so that the upper part of the following diagram (in which all other maps are as in~\eqref{eqn:correspondence-ind-res}) is cartesian:
		\[
		\begin{tikzcd}[column sep=5em]
			& & \fX \arrow[ld] \arrow[rd]& & \\
			& P \backslash \cN_P \arrow[ld] \arrow[rd]& & Q \backslash \cN_Q \arrow[ld] \arrow[rd] & \\
			L \backslash \cN_L & & G \backslash \cN_G & & M \backslash \cN_M.
		\end{tikzcd}
		\]
Denote by $\pi : \fX \to L \backslash \cN_L$ and $\rho : \fX \to M \backslash \cN_M$ the composition of maps from this diagram.

There is a natural map from $\fX$ to $P \backslash \mathrm{pt} \times_{G \backslash \mathrm{pt}} Q \backslash \mathrm{pt} \cong Q\backslash G/ P.$ 		
From the (finite) stratification of $Q\backslash G/ P$ by the double cosets $(Q\backslash G /P)_i:=Q\backslash Qg_iP/ P$, we deduce a stratification of $\fX$ by some substacks $\fX_i$. 

Fix $i$, and denote by $\pi_i, \rho_i$ the restrictions of $\pi, \rho$ to $\fX_i$.
Let also $\cN_i:=\cN_{Q\cap \leftindex^{g_i}{L}}\times_{\cN_{M\cap \leftindex^{g_i}{L}}} \cN_{M \cap \leftindex^{g_i}{P}},$
		and $G_i:=(Q\cap \leftindex^{g_i}{L})\times_{M\cap \leftindex^{g_i}{L}} (M \cap \leftindex^{g_i}{P})$, and consider
the commutative diagram
		\[
		\begin{tikzcd}[column sep=5em]
			L \backslash \cN_L & \fX_i \arrow[l,"\pi_i"]\arrow[d,"f_i"] \arrow[rd,"\rho_i"] & \\
			{}^{g_i}L \backslash \cN_{{}^{g_i}L} \arrow[u,"\mathrm{Ad}(g_i^{-1})"]& G_i \backslash \cN_i \arrow[l]\arrow[r]& M \backslash \cN_M
		\end{tikzcd}
		\]
		of~\cite[Theorem 2.5.3]{Dillery2024}.
		Here, by definition the correspondence in the lower line is defining $\ind_{M\cap {}^{g_i}L\subseteq M\cap {}^{g_i}P}^M \circ \res_{M\cap {}^{g_i}L\subseteq Q\cap {}^{g_i}L}^{^{g_i}L}$.
		Since $f_i$ satisfies the assumptions of \cite[Lemma 2.2.3]{Dillery2024}, the functor $f_{i!}f_i^*$ is isomorphic to the identity functor. Therefore, we have
		\begin{equation}\label{indres}
			(\rho_{i})_! (\pi_{i})^*\f=\ind_{M\cap {}^{g_i}L\subseteq M\cap {}^{g_i}P}^M\left(\res_{M\cap {}^{g_i}L\subseteq Q\cap {}^{g_i}L}^{^{g_i}L} \mathrm{Ad}(g_i^{-1})^*\f)\right).
		\end{equation}
In particular, $(\rho_{i})_! (\pi_{i})^*\f$ is perverse by Proposition \ref{indprop}\eqref{it:exactness}.
		
For any $i$, we now set $(Q\backslash G /P)_{\leq i}:=Q\backslash \bigcup_{j\leq i} Qg_iP/ P$.
 Due to the properties of the ordering chosen on the $g_i$ above we have that, for any $i$, $(Q\backslash G /P)_{i}$ is an open substack of $(Q\backslash G /P)_{\leq i}$ with closed complement $(Q\backslash G /P)_{\leq i-1}$. Pulling back these substacks along the map $\fX\rightarrow Q\backslash G/P$ we get substacks $\fX_{\leq i}, \fX_{\leq i-1}$.
The stack $\fX_{i}$ is then also an open substack of $\fX_{\leq i}$ with closed complement $\fX_{\leq i-1}$, which gives rise to a distinguished triangle
\[
(\rho_{\leq i-1})_! (\pi_{\leq i-1})^*\f \rightarrow (\rho_{\leq i})_! (\pi_{\leq i})^*\f \rightarrow \rho_{ i!}\pi_{ i}^*\f \xrightarrow{+1}.
\]
	
We set $\f_i:=\rho_{\leq i!}\pi_{\leq i}^*\f$. Since $\rho_{i!}\pi_{i}^*\f$ is perverse for each $i$ by~\eqref{indres}, using induction one sees that each $\f_i$ is also perverse, and that for any $i \geq 1$ the map $\f_{i-1} \to \f_i$ is injective. The triangles above thus become short exact sequences of perverse sheaves, and using~\eqref{indres} the corresponding filtration of $\f_s=\res_{M\subseteq Q}^G\circ \ind_{L\subseteq P}^G(\f)$ has the desired property.
	\end{proof}
	
\begin{remark}
Note that above we did not assume that $P$ and $Q$ were special. However this will be the case for most applications of Theorem~\ref{mackey}. If this is the case, then also each $M\cap {}^{g_i}L$ is special; this is the content of the assertion that special parabolic subgroups form an admissible class in the sense considered in Remark~\ref{rmk:parabolics}. For a justification, see the comments following~\cite[Definition~2.6.7]{Dillery2024}.
\end{remark}

\subsection{Decomposition into induction series}

We can now prove the first main result of this paper, which says that $\fR_{G,\bk}$ is partioned into induction series.
	
\begin{theorem}
\label{disj}
We have a partition
\[
\fR_{G,\bk}=\bigsqcup_{(L,\cO,\cE)\in \fL} \fR_{G,\bk}^{(L,\cO,\cE)}.
\]
\end{theorem}

	
	\begin{proof}
	The proof of is similar to that of the corresponding statement in case $G$ is connected given in~\cite{Achar2017a}.
	
		By a standard argument initially due to Bernstein in the setting of $p$-adic reductive groups, every simple perverse sheaf $\f\in \fR_{G,\bk}$ is part of an induction series. Indeed, consider a special parabolic subgroup with Levi subgroup $L$ such that $\res_{L\subseteq P}^G\f\neq 0$, and which is minimal for this property. Let also $\cG$ be a simple quotient of $\res_{L\subseteq P}^G\f$. Then, by Proposition~\ref{indprop}\eqref{it:transitivity} and the choice of $P$, $\res_{L\subseteq P}^G\f$ is annihilated by restriction to any Levi subgroup of a non-maximal parabolic subgroup of $L$, i.e.~is cuspidal. By exactness (see Proposition \ref{indprop}\eqref{it:exactness}), it follows that $\cG$ is cuspidal too. Setting $\cG=\IC(\cO,\cE)$, we therefore have a cuspidal triple $(L,\cO,\cE)$.
%
		By adjunction (see Proposition~\ref{indprop}\eqref{it:adjunction}), from the map $\res_{L\subseteq P}^G\f \twoheadrightarrow \cG$ we get a nonzero map $\f\maps \ind_{L\subseteq P}^G \cG$. Therefore, $\f \in \fR_{G,\bk}^{(L,\cO,\cE)}.$
		
For the disjointness of the induction series, assume $\cG \in \fR_{G,\bk}^{(L,\cO_L,\cE_L)}\cap \fR_{G,\bk}^{(M,\cO_M,\cE_M)}$. 
		Since induction series can be defined either in terms of simple submodules or simple quotients (see Corollary~\ref{cor:subquo}), $\cG$ is a simple quotient of $\ind_{L\subseteq P}^G(\IC(\cO_L,\cE_L))$ and also a simple submodule of $\ind_{M\subseteq Q}^G (\IC(\cO_M,\cE_M))$. The existence of the map $\ind_{L\subseteq P}^G(\IC(\cO_L,\cE_L))\twoheadrightarrow \cG \hookrightarrow \ind_{M\subseteq Q}^G (\IC(\cO_M,\cE_M))$ implies that
\[
\Hom_G(\ind_{L\subseteq P}^G(\IC(\cO_L,\cE_L)), \ind_{M\subseteq Q}^G (\IC(\cO_M,\cE_M)))\neq 0,
\]
and then by adjunction that
\[
\Hom_M(\res_{M\subseteq Q}^G\circ \ind_{L\subseteq P}^G(\IC(\cO_L,\cE_L)), \IC(\cO_M,\cE_M))\neq 0.
\]
		
Let $\f=\IC(\cO_L,\cE_L)$, and consider the notation of Theorem~\ref{mackey} for the present setting. Then there exists $i$ such that
\[
\Hom_M(\ind_{M\cap {}^{g_i}L\subseteq M\cap {}^{g_i}P}^M(\res_{M\cap {}^{g_i}L\subseteq Q\cap {}^{g_i}L}^{^{g_i}L} \mathrm{Ad}(g_i^{-1})^*\f)), \IC(\cO_M,\cE_M))\neq 0.
\]
By cuspidality of $\f$ (and thus of  $\mathrm{Ad}(g_i^{-1})^*\f$), this implies that the parabolic restriction is trivial, so ${}^{g_i}L \subset M$. We then have
\[
\Hom_M(\ind_{M\cap {}^{g_i}L\subseteq M\cap {}^{g_i}P}^M( \mathrm{Ad}(g_i^{-1})^*\f), \IC(\cO_M,\cE_M))\neq 0.
\]
By cuspidality of $\IC(\cO_M,\cE_M)$ this implies that ${}^{g_i} L = M$, and finally that $(L, \cO_L, \cE_L)$ and $(M,\cO_M,\cE_M)$ are equivalent.
	\end{proof}

	\section{Parametrization of induction series}
	\label{sec:parametrization}
	
Our next goal is to give, for a given cuspidal triple $(L,\cO,\cE)$ of $G$, a parametrization of $\fR_{G,\bk}^{(L,\cO,\cE)}$.
	
	\subsection{Weyl groups}
	\label{ss:Weyl-gps}

%
%
%
%

For any special Levi subgroup $L$ of $G$ and any $L$-orbit $\cO \subseteq \nn_L$, the group $\rN_G(L,\cO)$ acts on $\cO$, which induces an action of $\rN_G(L,\cO)/L$ on the set of isomorphism classes of simple $L$-equivariant local systems on $\cO$. This action can be described concretely in terms of representations of finite groups as follows. Recall that, after fixing $x \in \cO$, isomorphism classes of simple $L$-equivariant local systems on $\cO$ are in a canonical bijection with isomorphism classes of simple $\bk$-representations of $\rA_L(x)$. Now, as in the proof of Corollary~\ref{cor:subquo}, we have an identification $\rN_G(L)/L \cong \rA_{\rN_G(L)}(x) / \rA_L(x)$. Using the twist of representations as in~\S\ref{ss:simple-subobjs-ind-rep} for the normal subgroup $\rA_L(x)$ of $\rA_{\rN_G(L)}(x)$, we obtain an action of $\rA_{\rN_G(L)}(x) / \rA_L(x)$ on the set of isomorphism classes of simple $\bk$-representations of $\rA_L(x)$, which corresponds to the action considered above.

For a cuspidal triple $(L,\cO,\cE)$, we will denote by
\[
W_{(L,\cO,\cE)}
\]
the stabilizer of (the isomorphism class of) $\cE$ for the action of $\rN_G(L,\cO) / L$ on isomorphism classes of simple $L$-equivariant local systems on $\cO$.
This group is a subgroup of $\rN_G(L)/L$. When $G$ is connected this inclusion is in fact an equality, as follows from~\cite[Proposition~2.6 and Lemma~2.9]{Achar2017} and~\cite[Proposition 3.1]{Achar2017a}. This fact is not true in the disconnectedness case, as shown by the following example, which was communicated to us by Maarten Solleveld.

%
%
%
%
	
\begin{example}
\label{ex:weyl}
	Let $G=\mathrm{SL}_2(\bC)\rtimes \mathfrak{S}_3$ where elements of $\mathfrak{A}_3$ act trivially on $\mathrm{SL}_2(\bC)$ and elements of $\mathfrak{S}_3\setminus \mathfrak{A}_3$ act by $A\mapsto \leftindex^{t}{A}^{-1}$. Let also $\bk$ be a field which contains a nontrivial third root of unity. The special Levi subgroup $L \subset G$ whose identity component is the diagonal torus $T\subseteq \mathrm{SL}_2(\bC)$ is by definition $L=\rZ_G(T)=T\times \mathfrak{A}_3$, and we have $\rN_G(L)=\rN_{\mathrm{SL}_2(\bC)}(T)\rtimes \mathfrak{S}_3$, so that $\rN_G(L) / L \cong \mathfrak{S}_2 \times \mathfrak{S}_2$. (Here the left copy corresponds to $\rN_{\mathrm{SL}_2(\bC)}(T)/T$, while the right copy corresponds to $\mathfrak{S}_3/\mathfrak{A}_3$.)

Consider the nilpotent $L$-orbit $\cO=\{0\}$; here we have $\rA_L(0) = \mathfrak{A}_3$. Choose a non-trivial character $\chi: \mathfrak{A}_3\mapsto \bk^\times$, and let $\cE_{\chi}$ be the $L$-equivariant local system on $\cO$ corresponding to the associated $1$-dimensional representation. (Of course this local system is constant, but the equivariant structure is nontrivial.) Here we have $\rN_G(L,\cO) = \rN_G(L)$, but $W_{(L,\cO,\cE)} = \mathfrak{S}_2 \times \{1\}$.
	\end{example}
	
	\subsection{Parametrization}


Consider once again a cuspidal triple $(L,\cO,\cE)$. Applying Lem\-ma~\ref{lem:Fourier-cusp} to the pair $(\cO,\cE) \in \fR_{L,\bk}^\cusp$, we obtain an associated pair $(\cO',\cE') \in \fR_{L,\bk}^\cusp$.
	
\begin{theorem}
\phantomsection
\label{thm:parind}
%
\begin{enumerate}
\item
\label{it:parind-1}
There exists a canonical bijection between the set $\fR_{G,\bk}^{(L,\cO,\cE)}$ and the set of simple modules for the algebra $\en((\pi_{(L,\cO')})_* \widetilde{\cE'})$ of endomorphisms of the $L$-equivariant local system $(\pi_{(L,\cO')})_* \widetilde{\cE'}$.
\item
\label{it:parind-2}
The algebra $\en((\pi_{(L,\cO')})_* \widetilde{\cE'})$ is a crossed product of $W_{(L,\cO,\cE)}$ over the division algebra $\en(\cE)$; in case $\cE$ is absolutely irreducible, it is a twisted group algebra of $W_{(L,\cO,\cE)}$ over $\bk$.
\end{enumerate}
\end{theorem}

\begin{proof}
\eqref{it:parind-1}
As in the proof of Corollary~\ref{cor:subquo}, there exists a canonical bijection between $\fR_{G,\bk}^{(L,\cO,\cE)}$ and the set of simple subobjects of the $L$-equivariant local system $(\pi_{(L,\cO')})_* \widetilde{\cE'}$. Then we choose $x \in \cO'$, 	and apply Proposition~\ref{prop:bijection-simple-subreps}\eqref{it:bijection-simple-subreps}.

\eqref{it:parind-2}
Similarly the claim will follow from Proposition~\ref{prop:bijection-simple-subreps}\eqref{it:crossed-product} and Remark~\ref{rmk:twisted-gp-alg}, provided we show that $W_{(L,\cO,\cE)} = W_{(L,\cO',\cE')}$ and $\en(\cE) \cong \en(\cE')$. These isomorphisms are however clear by construction, since $\mathbb{T}_L$ is an equivalence of categories.
\end{proof}

\begin{remark}
\label{rmk:cocycle}
In the proof of Theorem~\ref{thm:parind} we have identified $\en((\pi_{(L,\cO')})_* \widetilde{\cE'})$ with the algebra $\en_{\rA_{\rN_G(L)}(x)_V}(\ind_{\rA_L(x)}^{\rA_{\rN_G(L)(x)_V}}(V))$ where $x \in \cO'$, $V$ is the representation of $\rA_L(x)$ corresponding to $\cE'$, and $\rA_{\rN_G(L)}(x)_V$ is its stabilizer in $\rA_{\rN_G(L)}(x)$. Assume that $V$ is absolutely irreducible, and extends to a representation of $\rA_{\rN_G(L)}(x)_V$. Then, after fixing such an extension, we obtain an identification $\ind_{\rA_L(x)}^{\rA_{\rN_G(L)(x)_V}}(V) \cong V \otimes \bk[W_{(L,\cO,\cE)}]$, and then a specific isomorphism $\en_{\rA_{\rN_G(L)}(x)_V}(\ind_{\rA_L(x)}^{\rA_{\rN_G(L)(x)_V}}(V)) \cong \bk[W_{(L,\cO,\cE)}]$. In case $G$ is connected, the results of~\cite[\S 3]{Achar2017}, translated in terms of representations of finite groups, say that (still assuming that $V$ is absolutely irreducible, i.e.~that $\cE$ is absolutely irreducible) there always exists a canonical extension of $V$ to a representation of $\rA_{\rN_G(L)}(x)_V$, hence a canonical isomorphism $\en_{\rA_{\rN_G(L)}(x)_V}(\ind_{\rA_L(x)}^{\rA_{\rN_G(L)(x)_V}}(V)) \cong \bk[W_{(L,\cO,\cE)}]$. (We moreover have $W_{(L,\cO,\cE)} = \rN_G(L)/L$ in this case, see the discussion preceding Example~\ref{ex:weyl}.)

This statement is not true in general in the disconnected case, since the algebra $\en_{\rA_{\rN_G(L)}(x)_V}(\ind_{\rA_L(x)}^{\rA_{\rN_G(L)(x)_V}}(V))$ might not be isomorphic to $\bk[W_{(L,\cO,\cE)}]$, see Example~\ref{nontrivial} below. We expect that this algebra can be described in canonical terms, but such a statement is still unknown, even in case $\mathrm{char}(\bk)=0$.
\end{remark}

	\subsection{An example of non-trivial cocycle}
	\label{ss:example-cocyle}

	

In this subsection we show trough an example that, in general, the algebra in Theorem~\ref{thm:parind} might \emph{not} be isomorphic to $\bk[W_{(L,\cO,\cE)}]$. This example is reproduced from~\cite[Example 3.2]{Aubert2018}, where it was discovered by considerations of the role of the generalized Springer correspondence in the Langlands correspondence. (This example is discussed in~\cite{Aubert2018} to illustrate a slightly different property, related to the discussion in Remark~\ref{rmk:parabolics}; the two properties are in fact identical in this case, essentially because the maximal special Levi subgroup of $G$ is $G^\circ$ for this group.)
	
	\begin{example}
	\label{nontrivial}
		Let $Q$ be the subgroup of $\mathrm{SL}_2(\bC)^5$ defined as:
		\begin{equation*}
		Q := \left\{ (\pm I_2) \times I_8, 
		 \begin{pmatrix} \pm i & 0 \\ 0 & \mp i \end{pmatrix}  \times I_4 \times -I_4, 
		\begin{pmatrix} 0 & \pm i \\ \pm i & 0 \end{pmatrix}  \times I_2 \times -I_2 \times I_2 \times -I_2,\begin{pmatrix} 0 & \mp 1 \\ \pm 1 & 0 \end{pmatrix}  \times I_2 \times -I_4 \times I_2 \right\}.
		\end{equation*}
		This group is isomorphic to the quaternion group of order 8. We define $G$ as the normalizer $G := \rN_{\mathrm{SL}_{10}(\bC)}(Q)$. Then we have
		\begin{gather*}
		G^\circ = \rZ_{\mathrm{SL}_{10}(\bC)}(Q) = \left( \rZ(\mathrm{GL}_2(\bC)) \times \mathrm{GL}_2(\bC)^4 \right) \cap \mathrm{SL}_{10}(\bC), \\
		G/G^\circ \cong (\bZ/2\bZ)^2, \\
		\rZ(G^\circ) = \left\{ (z_j)_{j=1}^{5} \in \rZ(\mathrm{GL}_2(\bC))^5 \left\mid \ \prod_{j=1}^{5} z_j^2 = 1 \right\}. \right.
		\end{gather*}
		Using the explicit representatives for $G/G^{\circ}$ given in~\cite[Example 3.2]{Aubert2018}, one can check that $\rZ_G(\rZ(G^{\circ})^\circ)=G^{\circ}$, so that the maximal special Levi subgroup for $G$ is $G^{\circ}$.
		
Let
\[
x = \mathrm{diag} \left( 0_2, \begin{pmatrix} 0 & 1 \\ 0 & 0 \end{pmatrix}, \begin{pmatrix} 0 & 1 \\ 0 & 0 \end{pmatrix}, \begin{pmatrix} 0 & 1 \\ 0 & 0 \end{pmatrix}, \begin{pmatrix} 0 & 1 \\ 0 & 0 \end{pmatrix} \right) \in \nn_G;
\]
we have
\[
\rA_{G^\circ}(x) \cong \bZ/2\bZ, \quad A_G(x) \cong Q.
\]
Let us assume $\cha ( \bk ) \neq 2$, so that $\rA_{G^\circ}(x)$ has a nontrivial $1$-dimensional representation. If $\cE$ is the corresponding $G^\circ$-equivariant local system on $\cO=G^\circ \cdot x$, then $(\cO,\cE)$ is a cuspidal pair for $G^\circ$.
		We therefore have the cuspidal triple $(G^{\circ}, \cO, \cE)$, with $W_{(G^{\circ}, \cO, \cE)} = G/G^\circ \cong (\bZ/2\bZ)^2$.
		
The perverse sheaf $\ind_{G^{\circ}\subseteq G^{\circ}}^G (\IC(\cO,\cE))$ identifies with the intersection cohomology complex of the pair $(G \cdot x,\mathcal{F})$ where $\mathcal{F}$ is the $G$-equivariant local system associated with the representation of $\rA_G(x)$ obtained by induction from the sign character of $\rA_{G^\circ}(x)$. The latter representation is either simple (but not absolutely simple) of dimension $4$, or the direct sum of two copies of a simple (and absolutely simple) $2$-dimensional representation. In any case, $\fR_{G,\bk}^{(G^{\circ}, \cO, \cE)}$ has only one element, while $(\bZ/2\bZ)^2$ has 4 isomorphism classes of simple representations, so that the corresponding twisted group algebra of $W_{(G^{\circ}, \cO, \cE)} \cong (\bZ/2\bZ)^2$ must be nontrivial. (In case $\ind_{G^{\circ}\subseteq G^{\circ}}^G (\IC(\cO,\cE))$ is not simple, this twisted group algebra is isomorphic to $\mathrm{M}_2(\bk)$.)

		
	\end{example}

%


	\subsection{An example}

We conclude this paper with an explicit description of our results in the case discussed in Example~\ref{ex:weyl}. Here there are two nilpotent orbits: the orbit of $0$, which satisfies $\rA_G(0)=\mathfrak{S}_3$, and the orbit of $x = \begin{pmatrix} 0 & 1 \\ 0 & 0 \end{pmatrix}$, where we have $\rA_G(x)=(\bZ/2\bZ) \times \mathfrak{S}_3$. There are also two conjugacy classes of special Levi subgroups, namely $T \times \mathfrak{A}_3$ and $G$. We have $\rN_G(T \times \mathfrak{A}_3) / (T \times \mathfrak{A}_3) \cong (\bZ/2\bZ)^2$. In all cases the cocycles from Theorem~\ref{thm:parind} will be trivial by the considerations in Remark~\ref{rmk:cocycle}.

\subsubsection{Characteristics not $2$ or $3$}
\label{sss:char-not-23}

Let us assume that $\mathrm{char}(\bk) \notin \{2,3\}$. Then $\mathfrak{S}_3$ has $3$ isomorphism classes of simple representations (those of the trivial representation, the sign representation, and a $2$-dimensional simple representation), giving rise to $3$ simple perverse sheaves denoted $(\IC_i : i \in \{1, 2, 3\})$ respectively, and $(\bZ/2\bZ) \times \mathfrak{S}_3$ has $6$ isomorphism classes of simple representations (obtained by tensoring the simple representations of $\mathfrak{S}_3$ either by the trivial representation or by the sign representation). We will denote by $(\IC_i : i \in \{4, 5, 6\})$ the simple perverse sheaves associated with the trivial representation of $\bZ/2\bZ$ (taking the representations of $\mathfrak{S}_3$ in the same order as above) and by $(\IC_i : i \in \{7, 8, 9\})$ those associated with the sign representation.

Let us assume that $\bk$ admits a nontrivial third root of unity. The special Levi subgroup $T \times \mathfrak{A}_3$ has a unique nilpotent orbit $\{0\}$, but $3$ cuspidal pairs corresponding to the $3$ isomorphism classes of simple $\mathfrak{A}_3$-representations. The $2$ nontrivial representations are conjugate under the action of the normalizer in $G$, so only one needs to be considered. For the trivial representation, the associated simple perverse sheaf is the skyscraper sheaf $\delta_0$ with trivial equivariant structure, and denoting by $B \subseteq \mathrm{SL}_2(\bC)$ the Borel subgroup of upper-triangular matrices we have
\[
\ind_{T \times \mathfrak{A}_3 \subset B \times \mathfrak{A}_3}^G(\delta_0) = \ind_{G^\circ \times \mathfrak{A}_3 \subseteq G^\circ \times \mathfrak{A}_3}^G(\ind_{T \subseteq B}^{G^\circ}(\delta_0) \boxtimes \bk).
\]
Here $\ind_{T \subseteq B}^{G^\circ}(\delta_0)$ is the Springer sheaf, which is the direct sum of the skyscraper sheaf $\delta_0$ (now seen as a perverse sheaf on $\nn_{G^\circ}$) and the simple perverse sheaf $\IC(G^\circ \cdot x, \underline{\bk})$. One easily checks using the considerations in the proof of Lemma~\ref{lem:ind-res-ss}\eqref{it:ind-exact-ss} that
\[
\ind_{G^\circ \times \mathfrak{A}_3 \subseteq G^\circ \times \mathfrak{A}_3}^G(\delta_0) \cong \IC_1 \oplus \IC_2, \quad \ind_{G^\circ \times \mathfrak{A}_3 \subseteq G^\circ \times \mathfrak{A}_3}^G(\IC(G^\circ \cdot x, \underline{\bk})) \cong \IC_4 \oplus \IC_5.
\]
This induction series therefore has $4$ elements, corresponding to the $4$ isomorphism classes of simple representations of $(\bZ/2\bZ)^2$.

Similarly, one sees that the induction series associated with a nontrivial character of $\mathfrak{A}_3$ consists of $\IC_3$ and $\IC_6$, corresponding to the $2$ simple representations of the stabilizer $\bZ/2\bZ$.

On the other hand $G$ has three cuspidal pairs, corresponding to the simple perverse sheaves $\IC_7$, $\IC_8$ and $\IC_9$, which form 3 other induction series.

\subsubsection{Characteristic $2$}

Let us now assume that $\mathrm{char}(\bk)=2$. Then $\mathfrak{S}_3$ has $2$ isomorphism classes of simple representations (the trivial representation and a $2$-dimensional simple representation), corresponding to two simple perverse sheaves $\IC_1$ and $\IC_2$ supported on $\{0\}$, and $(\bZ/2\bZ) \times \mathfrak{S}_3$ has $2$ simple representations (inflated from the simple representations of $\mathfrak{S}_3$), corresponding to two simple perverse sheaves $\IC_3$ and $\IC_4$ with full support.

Assuming again that $\bk$ admits a nontrivial third root of unity,
as above $T \times \mathfrak{A}_3$ has $3$ cuspidal pairs, but we need only consider $2$. For the trivial one, using the fact that the Springer sheaf has socle the skyscraper sheaf $\delta_0$, one can check that the corresponding induction series only contains $\IC_1$, corresponding to the unique simple representation of $(\bZ/2\bZ)^2$. For the other one the induction series only contains $\IC_2$, corresponding to the unique simple representation of $\bZ/2\bZ$. The simple perverse sheaves $\IC_3$ and $\IC_4$ are cuspidal, and form two other induction series.

\subsubsection{Characteristic $3$}

Finally, let us assume that $\mathrm{char}(\bk)=3$. Then $\mathfrak{S}_3$ has $2$ isomorphism classes of simple representations (the trivial and sign representations), corresponding to simple perverse sheaves $\IC_1$, $\IC_2$, and $(\bZ/2\bZ) \times \mathfrak{S}_3$ has $4$ isomorphism classes of simple representations, corresponding to simple perverse sheaves $\IC_3$, $\IC_4$, $\IC_5$, $\IC_6$ (using the same logic as above for the numbering).

The special Levi subgroup $T \times \mathfrak{A}_3$ has one cuspidal pair, corresponding to the skyscraper sheaf $\delta_0$ with trivial equivariant structure. The Springer sheaf is the sum of $2$ simple perverse sheaves as in~\S\ref{sss:char-not-23}, and one obtains that the associated induction series consists of $\IC_1$, $\IC_2$, $\IC_3$ and $\IC_4$, corresponding to the $4$ simple representations of $(\bZ/2\bZ)^2$. The simple perverse sheaves $\IC_5$ and $\IC_6$ are cuspidal, and form two other induction series.

	\printbibliography
	
\end{document}